\documentclass{amsart}
\usepackage[utf8]{inputenc}
\usepackage[hidelinks]{hyperref}
\usepackage{amsthm, amsmath, amssymb, enumitem, verbatim}

\newtheorem{proposition}{Proposition}

\theoremstyle{definition}
\newtheorem{definition}{Definition}
\newtheorem{example}{Example}

\theoremstyle{remark}
\newtheorem{remark}{Remark}

\DeclareMathOperator{\End}{End}
\DeclareMathOperator{\Hom}{Hom}
\DeclareMathOperator{\sgn}{sgn}

\begin{document}

\title[Multi-param. formal deformations of ternary hom-Nambu-Lie algebras]{Multi-parameter formal deformations of ternary hom-Nambu-Lie algebras}
\author{Per B\"ack}
\address{Division of Applied Mathematics, The School of Education, Culture and Communication, M\"alar\-dalen  University,  Box  883,  SE-721  23  V\"aster\r{a}s, Sweden,}
\email{per.back@mdh.se}


\begin{abstract}
In this note, we introduce a notion of multi-parameter formal deformations of ternary hom-Nambu-Lie algebras. Within this framework, we construct formal deformations of the three-dimensional Jacobian determinant and of the cross-product in four-dimensional Euclidean space. We also conclude that the previously defined ternary $q$-Virasoro-Witt algebra is a formal deformation of the ternary Virasoro-Witt algebra.
\end{abstract}

\maketitle

\section{Introduction}The notion of $n$-ary hom-Nambu-Lie algebras was introduced by Ataguema, Makhlouf, and Silvestrov in~\cite{AMS09}, including $n$-ary Nambu-Lie algebras as a special case. Generalizing an algebraic structure to a so-called hom-algebraic structure is often motivated by the fact that algebras that are rigid in terms of the former structure may be deformed when viewed as algebras in terms of the latter structure. For the simplest case $n=2$ of hom-Lie algebras, there seem to be many examples of formal deformations of Lie algebras into hom-Lie algebras (see e.g. \cite{Bac19, BR19, BRS18, MS10}). For the case $n=3$ of ternary hom-Nambu-Lie algebras, examples of formal deformations seem to be lacking, however. The purpose of this short, but self-contained note is to provide examples of formal deformations of ternary Nambu-Lie algebras into ternary hom-Nambu-Lie algebras, to devise a concrete method on how to construct these, and to define a framework in which they all fit.

\section{Preliminaries}
We denote by $\mathbb{N}$ the natural numbers including zero, and by $K$ a field of characteristic zero. The symmetric group on $p$ letters is written $S_p$.

\begin{definition}[Ternary hom-Nambu algebra] A \emph{ternary hom-Nambu algebra}, written $(V, [\cdot,\cdot,\cdot], (\alpha,\beta))$, consists of a $K$-vector space $V$, two $K$-linear maps $\alpha,\beta\colon V\to V$ called \emph{twisting maps}, and a $K$-trilinear map $[\cdot,\cdot,\cdot]\colon V\times V\times V\to V$ called the \emph{hom-Nambu bracket}. For all $x_1,\dots,x_5\in V$, they are required to satisfy the \emph{ternary hom-Nambu identity}:
\begin{align*}
&[\alpha(x_1),\beta(x_2),[x_3,x_4,x_5]]=[[x_1,x_2,x_3],\alpha(x_4),\beta(x_5)]\\
&+ [\alpha(x_3),[x_1,x_2,x_4],\beta(x_5)]+[\alpha(x_3),\beta(x_4),[x_1,x_2,x_5]].
\end{align*}
\end{definition}

\begin{definition}[Ternary hom-Nambu-Lie algebra] A \emph{ternary hom-Nam\-bu-Lie algebra} is a ternary hom-Nambu algebra $(V,[\cdot,\cdot,\cdot], (\alpha,\beta))$ where the bracket $[\cdot,\cdot,\cdot]$ is skew-symmetric, meaning that for all $x_1,x_2,x_3\in V$, $\sigma\in S_3$, 
\begin{equation*}
[x_{\sigma(1)},x_{\sigma(2)}, x_{\sigma(3)}]=\sgn(\sigma) [x_1,x_2,x_3].
\end{equation*}
\end{definition}

\begin{remark}If we put $\alpha=\beta=\mathrm{id}_V$ in the above definitions, we get a \emph{ternary Nambu algebra} and a \emph{ternary Nambu-Lie algebra}, respectively. Instead of writing $(V,[\cdot,\cdot,\cdot], (\mathrm{id}_V,\mathrm{id}_V))$, we then use the shorthand notation $(V,[\cdot,\cdot,\cdot])$.
\end{remark}

\begin{definition}[Morphisms of ternary hom-Nambu-(Lie) algebras] A \emph{morphism} from a ternary hom-Nambu-(Lie) algebra $A:=(V,[\cdot,\cdot,\cdot], (\alpha,\beta))$ to a ternary hom-Nambu-(Lie) algebra $A':=(V',[\cdot,\cdot,\cdot]', (\alpha',\beta'))$ is a $K$-linear map $f\colon V\to V'$ such that for all $x_1,x_2,x_3\in V$, $f([x_1,x_2,x_3])=[f(x_1),f(x_2),f(x_3)]'$, $f\circ\alpha = \alpha'\circ f$, and $f\circ\beta=\beta'\circ f$. We denote the set of all such maps by $\Hom_K(A,A')$, and put $\End_K(A):=\Hom_K(A,A)$ for the set of \emph{endomorphisms}.
\end{definition}

A ternary hom-Nambu-(Lie) algebra $A:=(V,[\cdot,\cdot,\cdot],(\alpha,\beta))$ is \emph{multiplicative} whenever $\alpha=\beta\in\End_K(A)$.

\begin{proposition}[\cite{AMS09}]\label{prop:morphism-construction} Let $A:=(V,[\cdot,\cdot,\cdot])$ be a ternary Nambu(-Lie) algebra and $\rho\in\End_K(A)$. Then $(V,\rho\circ[\cdot,\cdot,\cdot], (\rho,\rho))$ is a multiplicative, ternary hom-Nambu(-Lie) algebra.
\end{proposition}

Let us now take a look at some classical examples of ternary (hom-)Nambu-Lie algebras that can be found in the literature.

\begin{example}[The cross-product in $\mathbb{R}^4$]\label{ex:vector-prod}Denote by $\mathbb{R}^4$ be the four-dimensional Euclidean vector space over $\mathbb{R}$, and by $\{\vec{e}_1,\vec{e}_2,\vec{e}_3,\vec{e}_4\}$ its standard basis. Using the Einstein summation convention, let $\vec{x}:=x^i\vec{e}_i$,  $\vec{y}:=y^i\vec{e}_i$, and $\vec{z}:= z^i\vec{e}_i$ where $x^i,y^i,z^i\in\mathbb{R}$ are arbitrary and $1\leq i\leq 4$. We then get a ternary Nambu-Lie algebra $(\mathbb{R}^4,[\cdot,\cdot,\cdot])$ over $\mathbb{R}$ by defining the bracket as the cross-product in $\mathbb{R}^4$,
\begin{equation*}
[\vec{x},\vec{y},\vec{z}]:=\vec{x}\times \vec{y}\times \vec{z}:=\begin{vmatrix}x^1&y^1&z^1&\vec{e}_1\\x^2&y^2&z^2&\vec{e}_2\\x^3&y^3&z^3&\vec{e}_3\\ x^4&y^4&z^4&\vec{e}_4\end{vmatrix}.
\end{equation*}
\end{example}

\begin{example}[The three-dimensional Jacobian determinant~\cite{AMS09}]\label{ex:jacobian}Let $q_1, q_2,$ and $q_3$ be elements of the polynomial $K$-algebra $K[x_1,x_2,x_3]$ in three indeterminates $x_1, x_2,$ and $x_3$. Denote by $q$ the triplet $(q_1,q_2,q_3)$, by $x$ $(x_1,x_2,x_3)$, and by $J(q(x))=(\partial q_i/ \partial x_j)_{1\leq i,j\leq 3}$, the three-dimensional Jacobian. By defining 
\begin{equation*}
[q_1(x),q_2(x),q_3(x)]:=\det(J(q(x)))=\begin{vmatrix}\frac{\partial q_1}{\partial x_1}&\frac{\partial q_1}{\partial x_2}&\frac{\partial q_1}{\partial x_3}\\ \frac{\partial q_2}{\partial x_1}& \frac{\partial q_2}{\partial x_2}& \frac{\partial q_2}{\partial x_3}\\ \frac{\partial q_3}{\partial x_1}& \frac{\partial q_3}{\partial x_2} &\frac{\partial q_3}{\partial x_3}\end{vmatrix},
\end{equation*}
we get a ternary Nambu-Lie algebra $(K[x_1,x_2,x_3],[\cdot,\cdot,\cdot])$ over $K$. Now, any $\rho\in\End_K((K[x_1,x_2,x_3],[\cdot,\cdot,\cdot]))$ gives rise to a ternary hom-Nambu-Lie algebra by Proposition \ref{prop:morphism-construction}. Let $\gamma_1,\gamma_2,\gamma_3\in K[x_1,x_2,x_3]$, $\gamma:=(\gamma_1,\gamma_2,\gamma_3)$, and furthermore define $\rho_\gamma\colon K[x_1,x_2,x_3]\to K[x_1,x_2,x_3]$ by $q(x)\mapsto q(\gamma)$. As $J(\rho_\gamma(q_1),\rho_\gamma(q_2),\rho_\gamma(q_3))$ $=J(q(\gamma))J(\gamma(x))$, 
$\det(J(\rho_\gamma(q_1),\rho_\gamma(q_2),\rho_\gamma(q_3)))=\det(J(q(\gamma)))\det (J(\gamma(x)))$. We have that $\det(J(q(\gamma))=\rho_\gamma\left(\det (J(q(x)))\right)$, so by means of Proposition \ref{prop:morphism-construction}, we have a ternary hom-Nambu-Lie algebra $(K[x_1,x_2,x_3],[\cdot,\cdot,\cdot]_\gamma,(\rho_\gamma,\rho_\gamma))$ where $[\cdot,\cdot,\cdot]_\gamma:=\rho_\gamma\circ[\cdot,\cdot,\cdot]$, if $\det(J(\gamma(x)))=1$.
\end{example}

\begin{example}[The ternary $q$-Virasoro-Witt algebra \cite{AMS10}]\label{ex:vir-witt} Let $W$ be a vector space over $\mathbb{C}$ with generating set $\{Q_n,R_n\}_{n\in\mathbb{Z}}$, and $[\cdot,\cdot,\cdot]\colon W\times W\times W\to W$ the $\mathbb{C}$-trilinear, skew-symmetric bracket defined on the generators of $W$ by
\begin{align*} 
[Q_k,Q_m,Q_n]=&(k-m)(m-n)(k-n)R_{k+m+n},\\
[Q_k,Q_m,R_n]=&(k-m)(Q_{k+m+n}+znR_{k+m+n}),\\
[Q_k,R_m,R_n]=&(n-m)R_{k+m+n},\\
[R_k,R_m,R_n]=&0.
\end{align*}
Whenever $z=\pm 2i$, these relations define a ternary Nambu-Lie algebra $(W,[\cdot,\cdot,\cdot])$ over $\mathbb{C}$ called the \emph{ternary Virasoro-Witt algebra}, introduced by Curtright, Fairlie and Zachos in~\cite{CFZ08}. Ammar, Makhlouf, and Silvestrov found in~\cite{AMS10} that a $\mathbb{C}$-linear map $\rho_q\colon W\to W$ defined on the generators of $W$ by $\rho_q(Q_n):=q^nQ_n$ and $\rho_q(R_n):=q^nR_n$ for some $q\in\mathbb{C}$ is an endomorphism of the ternary Virasoro-Witt algebra. Hence, by Proposition \ref{prop:morphism-construction}, we have a ternary hom-Nambu-Lie algebra $(W,[\cdot,\cdot,\cdot]_q, (\rho_q,\rho_q))$, $[\cdot,\cdot,\cdot]_q:=\rho_q\circ[\cdot,\cdot,\cdot]$ whenever $z=\pm 2i$. This algebra is called the \emph{ternary $q$-Virasoro-Witt algebra}, including  the ternary Virasoro-Witt algebra in the case $q=1$. When $q\neq1$, one does in general not get a ternary Nambu-Lie algebra.
\end{example}

\begin{remark}When $z\neq \pm2i$, the ternary Virasoro-Witt algebra is in fact a hom-Nambu-Lie algebra \cite{AMS10}. 
\end{remark}

\section{Examples}
Guided by the way of constructing the ternary $q$-Virasoro-Witt algebra in Example \ref{ex:vir-witt} by means of Proposition \ref{prop:morphism-construction}, we will here apply the same method to the cross-product in $\mathbb{R}^4$ defined in Example \ref{ex:vector-prod} and the three-dimensional Jacobian determinant defined in Example \ref{ex:jacobian}.

\begin{example}[A deformed cross-product in $\mathbb{R}^4$]\label{ex:def-vector-prod}Using the same notation as in Example \ref{ex:vector-prod}, we would like to find all $\mathbb{R}$-linear maps $\rho\colon\mathbb{R}^4\to \mathbb{R}^4$ such that $\rho\left([\vec{x},\vec{y},\vec{z}]\right)=[\rho(\vec{x}),\rho(\vec{y}),\rho(\vec{z})]$. To this end, let $\rho(\vec{e}_l)= a_l^i\vec{e}_i$ where $a_l^i\in\mathbb{R}$, $1\leq i,l\leq 4$. By using the Kronecker delta $\delta_j^i$ and a Levi-Civita symbol $\varepsilon_{ijk}^l$,
\begin{align*}
\varepsilon_{ijk}^l:=&\begin{cases}+1&\text{if $(i,j,k,l)$ is an even permutation of $(1,2,3,4)$},\\-1&\text{if $(i,j,k,l)$ is an odd permutation of $(1,2,3,4)$},\\\phantom{\pm}0&\text{if any two indices are equal,} \end{cases}
\end{align*}

\begin{align*}
\rho\left([\vec{e}_l,\vec{e}_m,\vec{e}_n]\right)=&\rho\left(\begin{vmatrix}\delta^1_l&\delta^1_m&\delta^1_n&\vec{e}_1\\\delta^2_l&\delta^2_m&\delta^2_n&\vec{e}_2\\\delta^3_l&\delta^3_m&\delta^3_n&\vec{e}_3\\ \delta^4_l&\delta^4_m&\delta^4_n&\vec{e}_4\end{vmatrix}\right)=\rho\left(\varepsilon_{pqr}^s\delta^p_l\delta^q_m\delta^r_n\vec{e}_s\right)=\rho\left(\varepsilon_{lmn}^s\vec{e}_s\right)\\
=&\varepsilon_{lmn}^s\rho(\vec{e}_s)=\varepsilon_{lmn}^sa_s^t\vec{e}_t,\\
[\rho(\vec{e}_l),\rho(\vec{e}_m),\rho(\vec{e}_n)]=&\begin{vmatrix}a_l^1&a_m^1&a_n^1&\vec{e}_1\\ a_l^2&a_m^2&a_n^2&\vec{e}_2\\ a_l^3&a_m^3&a_n^3&\vec{e}_4\\ a_l^4&a_m^4&a_n^4&\vec{e}_4\\ \end{vmatrix} = \varepsilon^s_{pqr}a_l^pa_m^qa_n^r\vec{e}_s.
\end{align*}
By comparing coefficients, we get the following system of equations: $\varepsilon^s_{lmn}a^t_s=\varepsilon^t_{pqr}a_l^pa_m^qa_n^r$, $1\leq l,m,n,t\leq 4$. 
A solution $\rho_\theta$, $\theta:=(\theta_1,\theta_2)$, is $a_1^1=a_3^3=\cos\theta_1$, $a_2^2=a_4^4=\cos\theta_2$, $a_1^3=-a_3^1=\sin\theta_1$, $a_2^4=-a_4^2=\sin\theta_2$, as the only nonzero elements, $\theta_1,\theta_2\in\mathbb{R}$. In matrix form,

\begin{align*}
\rho_{\theta}=&\begin{pmatrix}a_1^1 &a_2^1& a_3^1 & a_4^1\\ a_1^2 &a_2^2& a_3^2 & a_4^2\\ a_1^3 &a_2^3& a_3^3 & a_4^3\\ a_1^4 &a_2^4& a_3^4 & a_4^4\end{pmatrix} = \begin{pmatrix}\cos \theta_1&0&-\sin \theta_1&0\\ 0&\cos \theta_2&0&-\sin \theta_2\\ \sin \theta_1&0&\cos \theta_1&0 \\ 0&\sin \theta_2&0&\cos \theta_2\end{pmatrix}\\
=&\begin{pmatrix}\cos\theta_1&0&-\sin\theta_1&0\\0&1&0&0\\\sin\theta_1&0&\cos\theta_1&0\\0&0&0&1\end{pmatrix}\begin{pmatrix}1&0&0&0\\0&\cos\theta_2&0&-\sin\theta_2\\0&0&1&0\\0&\sin\theta_2&0&\cos\theta_2\end{pmatrix}.
\end{align*}
Thus, we see that $\rho_\theta$ is the product of two rotation matrices (which commute): one in the $\vec{e}_1\vec{e}_3$-plane by an angle $\theta_1$, and one in the $\vec{e}_2\vec{e}_4$-plane by an angle $\theta_2$. We put $[\cdot,\cdot,\cdot]_\theta:=\rho_\theta\circ[\cdot,\cdot,\cdot]$ and call the hom-Nambu-Lie algebra $(\mathbb{R}^4,[\cdot,\cdot,\cdot]_\theta,(\rho_\theta,\rho_\theta))$ a \emph{deformed cross-product in $\mathbb{R}^4$} (in the last section, we will justify the name \emph{deformed}). In general, this is not a Nambu-Lie algebra. Put e.g. $\theta_1=\theta_2=\pi/2$ and $\vec{e}_5:=\vec{e}_1+\vec{e}_2+\vec{e}_4.$ Then $[\vec{e}_1,\vec{e}_2,[\vec{e}_3,\vec{e}_4,\vec{e}_5]_{\theta}]_{\theta}=\vec{e}_1+\vec{e}_2$, while $[[\vec{e}_1,\vec{e}_2,\vec{e}_3]_\theta,\vec{e}_4,\vec{e}_5]_\theta+[\vec{e}_3,[\vec{e}_1,\vec{e}_2,\vec{e}_4]_\theta,\vec{e}_5]_\theta+[\vec{e}_3,\vec{e}_4,[\vec{e}_1,\vec{e}_2,\vec{e}_5]_\theta]_\theta=-\vec{e}_1-\vec{e}_2$.

\end{example}

\begin{example}[A deformed three-dimensional Jacobian determinant]\label{ex:def-jacobian}Using the same notation as in Example \ref{ex:jacobian}, we would like to find a nontrivial $\gamma(x)$ with $\det(J(\gamma(x)))=1$. We start with the simple, but nontrivial assumption that $J(\gamma(x))$ is an upper triangular matrix (the case when $J(\gamma(x))$ is a lower triangular matrix is analogous). Using that $\det(J(\gamma(x)))$ is the product of all the diagonal entries of $J(\gamma(x))$, one readily verifies that $\det(J(\gamma(x)))=1$ if and only if $\gamma(x_1,x_2,x_3)=(k_1x_1+p_1(x_2,x_3), k_2x_2+p_2(x_3), k_3x_3+k_4)$ for some $p_1(x_2,x_3)\in K[x_2,x_3]$, $p_2(x_2)\in K[x_2]$, and $k_1,k_2,k_3,k_4\in K$ where $k_1k_2k_3=1$. A basis of $K[x_1,x_2,x_3]$ as a $K$-vector space consists of all monomials $x_1^lx_2^mx_3^n$ where $l,m,n\in\mathbb{N}$, and so if we define $\rho_\gamma(x_1^lx_2^mx_3^n):=(k_1x_1+p_1(x_2,x_3))^l(k_2x_2+p_2(x_3))^m(k_3x_3+k_4)^n$ and then extend the definition linearly, we have a ternary hom-Nambu-Lie algebra $(K[x_1,x_2,x_3],[\cdot,\cdot,\cdot]_\gamma, (\rho_\gamma,\rho_\gamma))$ where $[\cdot,\cdot,\cdot]_\gamma:=\rho_\gamma\circ[\cdot,\cdot,\cdot]$. We refer to it as a \emph{deformed three-dimensional Jacobian determinant} (again, we will justify the name \emph{deformed} in the last section). In general this is not a Nambu-Lie algebra. Take e.g. $p_1(x_2,x_3)=p_2(x_3)=0$, $k_1=k_2=k_3=1$, and $q_1:=x_1, q_2:=x_2, q_3:=x_3^3, q_4:=x_1^2, q_5:=x_2x_3$. Then $[q_1,q_2,[q_3,q_4,q_5]_\gamma]_\gamma=18x_1(x_3+2k_4)^2$, while $[[q_1,q_2,q_3]_\gamma,q_4,q_5]_\gamma+[q_3,[q_1,q_2,q_4]_\gamma,q_5]_\gamma+[q_3,q_4,[q_1,q_2,q_5]_\gamma]_\gamma=6x_1(x_3+k_4)(3x_3+5k_4)$, so the two expressions are equal if and only if $k_4=0$, in which case we have the original three-dimensional Jacobian determinant.
\end{example}

\section{Multi-parameter formal deformations}
One-parameter formal ternary hom-Nambu-Lie deformations were defined in~\cite{AMM11}. Here, we generalize that notion to multi-parameter analogues. 
\begin{definition}[Multi-parameter formal ternary hom-Nambu(-Lie) deformation]\label{def:multi-param-def}An \emph{$n$-parameter formal ternary hom-Nambu(-Lie) deformation} of a ternary hom-Nam\-bu(-Lie) algebra $(V,[\cdot,\cdot,\cdot]_0,(\alpha_0,\beta_0))$ over $K$ is a ternary hom-Nambu(-Lie) algebra $(V[[t_1,t_2,\ldots,t_n]], [\cdot,\cdot,\cdot]_t,(\alpha_t,\beta_t))$ over $K[[t_1,t_2,\ldots,t_n]]$ where $n\in\mathbb{N}_{>0}$, $t:=(t_1,t_2,\ldots,t_n)$, and
\begin{equation*}
[\cdot,\cdot,\cdot]_t=\sum_{i\in\mathbb{N}^n} [\cdot,\cdot,\cdot]_i t^i,\quad \alpha_t=\sum_{i\in\mathbb{N}^n} \alpha_it^i,\quad \beta_t=\sum_{i\in\mathbb{N}^n} \beta_it^i.
\end{equation*}
Here, $i:=(i_1,i_2,\ldots,i_n)\in\mathbb{N}^n$, $t^i:=t_1^{i_1}t_2^{i_2}\cdots t_n^{i_n},$ and $[\cdot,\cdot,\cdot]_i\colon V\times V\times V\to V$ is a $K$-trilinear operation, $\alpha_i,\beta_i\colon V\to V$ two $K$-linear maps, extended homogeneously to a $K[[t_1,t_2,\ldots,t_n]]$-trilinear operation, $[\cdot,\cdot,\cdot]_i\colon V[[t_1,t_2,\ldots,t_n]]\times V[[t_1,t_2,\ldots,t_n]]\times V[[t_1,t_2,\ldots,t_n]]\to V[[t_1,t_2,\ldots,t_n]]$, and two $K[[t_1,t_2,\ldots,t_n]]$-linear maps $\alpha_i,\beta_i\colon V[[t_1,t_2,\ldots,t_n]]\to V[[t_1,t_2,\ldots,t_n]]$, respectively.
\end{definition}

In the above definition, a map $f\colon V\to V$ is said to be \emph{extended homogeneously} to a map $f\colon V[[t_1,t_2,\ldots,t_n]]\to V[[t_1,t_2,\ldots,t_n]]$ when $f(at^i):=f(a)t^i$ for all $a\in V$ and $i\in\mathbb{N}^n$. The case for ternary maps is analogous.

\begin{remark}Note that there is some slight abuse of notation in Definition \ref{def:multi-param-def}. The maps $[\cdot,\cdot,\cdot]_0$, $\alpha_0$, and $\beta_0$ are the maps $[\cdot,\cdot,\cdot]_i$, $\alpha_i$, and $\beta_i$ where $i=(0,0,\ldots,0)\in\mathbb{N}^n$.
\end{remark} 

\begin{proposition}The ternary $q$-Virasoro-Witt algebra is a one-parameter formal ternary hom-Nambu-Lie deformation of the ternary Virasoro-Witt algebra.
\end{proposition}

\begin{proof}Put $t:=q-1$ and regard it as a formal parameter.
\end{proof}

\begin{proposition}The deformed cross-product is a two-parameter formal ternary hom-Nambu-Lie deformation of the cross-product in $\mathbb{R}^4$.
\end{proposition}

\begin{proof}Put $t_1:=\theta_1$ and $t_2:=\theta_2$ and regard them as formal parameters. Replace $\cos t_k$ and $\sin t_k$ for $1\leq k\leq 2$ with their corresponding formal power series $\sum_{i=0}^\infty \frac{(-1)^i}{(2i)!}t_k^{2i}$ and $\sum_{i=0}^\infty\frac{(-1)^i}{(2i+1)!}t_k^{2i+1}$, respectively.
\end{proof}

\begin{proposition}The deformed three-dimensional Jacobian determinant is a multi-parameter formal ternary hom-Nambu-Lie deformation of the three-dimensional Jacobian determinant.
\end{proposition}

\begin{proof}Choose $k_1,k_2,k_3$ such that $k_1k_2k_3=1$ and regard $k_4$ and all the coefficients in the polynomials $p_1(x_2,x_3)$, $p_2(x_3)$ as formal parameters.
\end{proof}

\section{Acknowledgments}
I wish to thank Joakim Arnlind for some initial discussions.

\end{document}